\documentclass{amsart} 
\usepackage{graphics,amsmath,amsthm,amssymb} 
\usepackage{color} 
\newtheorem{theorem}{Theorem}[section]

\newtheorem{remark}[theorem]{Remark} 
 
\newtheorem{example}[theorem]{Example}

\begin{document} 
\title{Partition theorems and the Chinese remainder theorem} 
\author[S.-C. Chen]{Shi-Chao Chen} 
\address{Institute of Contemporary Mathematics, Department of Mathematics 
and Statistics Sciences, Henan University, Kaifeng, 475001, 
China, schen@henu.edu.cn} 
\begin{abstract}The famous partition theorem of Euler states that partitions of $n$ into distinct parts are equinumerous with partitions of $n$ into odd parts. Another famous partition theorem due to MacMahon states that the number of partitions of $n$ with all parts repeated at least once equals the number of partitions of $n$ where all parts must be even or congruent to $3 \pmod 6$. These partition theorems were further extended by Glaisher, Andrews, Subbarao, Nyirenda and Mugwangwavari. In this paper, we utilize the Chinese remainder theorem to prove a comprehensive partition theorem that encompasses all existing partition theorems. We also give a natural generalization of Euler's theorem  based on a special complete residue system. Furthermore, we establish interesting congruence connections between the partition function $p(n)$ and related partition functions. 
\end{abstract}

\maketitle 
\section{Introduction} 
Let $\mathbb{N}=\{1, 2, \cdots\}$ be the set of natural numbers. A partition of $n\in\mathbb{N}$ is a finite non-increasing sequence of positive integers $\lambda=(\lambda_1,\lambda_2,\cdots,\lambda_s)$ such that 
\[n=\lambda_1+\lambda_2+\cdots+\lambda_s\] 
with $\lambda_1\ge\lambda_2\ge\cdots\ge\lambda_s\ge1$. 
We refer to the numbers $\lambda_i$ as the \emph{parts} of $\lambda$. If a part $\lambda_i$ appears $d$ times in $\lambda$, then we say the \emph{multiplicity} of $\lambda_i$ is $d$. For instance, a partition of $8$ is $\lambda=(2, 2, 2, 1, 1)$, the multiplicity of the part $2$ is 3 and the multiplicity of $1$ is 2. For more details about partitions, we refer the reader to \cite{an}. 

Let $\mathcal{A}, \mathcal{B}$ be two subsets of $\mathbb{N}$ and $n$ be a positive integer. 
We denote by 
$P(\mathcal{A}; n)$ the number of partitions of $n$ in which the multiplicities of the parts belong to $\mathcal{A}$ and by $Q(\mathcal{B}; n)$ the number of partitions of $n$ whose parts belong to $\mathcal{B}$. For convenience, we define $P(\mathcal{A}; 0)=Q(\mathcal{B}; 0)=1$. 
Many classical partition theorems in the theory of partitions can be expressed as identities of the form 
\begin{equation}\label{identity} 
P(\mathcal{A}; n) = Q(\mathcal{B}; n). 
\end{equation} 

One of the most famous is Euler's partition theorem, which states that the number of partitions of $n$ into distinct parts are equinumerous with the number of partitions of $n$ into odd parts \cite[Corollary 1.2]{an}. Since Euler, many different and unexpected partition identities were found \cite{ald, an}. 

In this paper we are concerned with different extension of Euler's theorem. To give more explicit statements of such results, we need some notation. 
For $a\in \mathbb{N}$ and $b\in\mathbb{N}\cup\{0\}$, we let $a\mathbb{N}+b$ denote the residue class $b\pmod a$, i.e. 
\[a\mathbb{N}+b=\{an+b: n\in\mathbb{N}\},\] 
and define $a\mathbb{N}+b:=\emptyset$ if $a$ tends to infinity. 
Let $a\mathbb{N}\setminus b\mathbb{N}$ denote the substraction of $b\mathbb{N}$ from $a\mathbb{N}$, i.e. $a\mathbb{N}\setminus b\mathbb{N}=\{x: x\in a\mathbb{N}\text{ and } x\not\in b\mathbb{N}\}$. For instance, 
$2\mathbb{N}\setminus 3\mathbb{N}=(6\mathbb{N}+2)\cup(6\mathbb{N}+4)$. 
With our notation, Euler's theorem can be restated as follows: 
\begin{theorem}[Euler]\label{Euler} For all $n\in\mathbb{N}$, we have 
\[P(\{1\}; n) = Q(\mathbb{N}\setminus 2\mathbb{N}; n).\] 
\end{theorem} 
Euler's theorem has a natural generalization due to Glaisher \cite[Corollary 1.3]{an}, where $\mathcal{A}=\{1\}$ in Theorem \ref{Euler} was replaced by $\mathcal{A}=\{1, 2, \cdots, d-1\}$. 
\begin{theorem}[Glaisher]\label{Glaisher} For any integer $d\ge2$, we have 
\[P(\{1, 2, \cdots, d-1\}; n) = Q\big(\mathbb{N}\setminus d\mathbb{N}; n\big). \] 
\end{theorem} 
Analogous to Euler's theorem, MacMahon \cite[p.54]{ma} proved another elegant partition theorem with $\mathcal{A}=\mathbb{N}\setminus\{1\}$, which states that the number of partitions of $n$ wherein no part appears with multiplicity one equals the number of partitions of $n$ where all parts must be even or congruent to $3 \pmod 6$. That is 
\begin{theorem}[MacMahon]\label{macmahon} We have 
\[P(\mathbb{N}\setminus\{1\}; n) = Q\bigg(2\mathbb{N}\cup (3\mathbb{N}\setminus 6\mathbb{N}); n\bigg). \] 
\end{theorem} 
MacMahon's theorem has many generalizations. In 1967, Andrews \cite{an1} generalized MacMahon's Theorem with $\mathcal{A}=\mathbb{N}\setminus\{1,3, \cdots,2r-1\}$ for a positive integer $r$. 
\begin{theorem}[Andrews]\label{andrews} 
For any $r\in \mathbb{N}$, we have 
\[P(\mathbb{N}\setminus\{1,3, \cdots,2r-1\}; n) = Q\bigg(2\mathbb{N}\cup ((2r+1)\mathbb{N}\setminus(4r+2)\mathbb{N}); n\bigg). \] 
\end{theorem} 
A further improvement of Andrews' theorem was made by Subbarao \cite{su} in 1971, where the set $\mathcal{A}$ is a union of two sets, one is a subset of odd numbers and the other is a subset of even numbers. 
\begin{theorem}[Subbarao]\label{subbarao} 
Let $l > 1, r \ge 0$ be integers. Then 
\[P(\{2i+j(2r+1): i=0, 1, \cdots, l-1; j=0, 1\}; n)\] 
\[ = Q\bigg((2\mathbb{N}\setminus 2l\mathbb{N})\cup ((2r+1)\mathbb{N}\setminus(4r+2)\mathbb{N}); n\bigg). \] 
\end{theorem} 
Note that Subbarao's theorem reduces to Andrews' theorem as $l$ tends to infinity and $2l\mathbb{N}$ is replaced by $\emptyset$ in Theorem \ref{subbarao}. 

Subbarao's theorem was covered by Nyirenda and Mugwangwavari in their recent work \cite{nm}. They found the set $\mathcal{A}$ may be a union of $m$ sets, each belongs to one of residue classes modulo $m$. 
\begin{theorem}[Nyirenda and Mugwangwavari]\label{nm} 
Assume that $l,\, r$, $a$ and $p$ are positive integers such that 
gcd$(a, p) = 1$. Then 
\[P(\{pi+j(pr+a): i=0, 1, \cdots, l-1; j=0, 1,\cdots,p-1\}; n)\] 
\[= Q\bigg((p\mathbb{N}\setminus lp\mathbb{N})\bigcup ((pr+a)\mathbb{N}\setminus p(pr+a)\mathbb{N}); n\bigg). \] 
\end{theorem} 
In this paper we reveal a general principle underlying Theorems 1.1-1.6, i.e. there is a correspondence between partition identities   (\ref{identity}) and identities of some geometric series (see Theorem \ref{lemma} below). With this principle, we are able to utilize the Chinese remainder theorem to establish a more comprehensive partition theorem, which implies Theorem 1.1-1.6. 
\begin{theorem}\label{mainth}Let $k, l$ and $s$ be positive integers. 
Suppose that $m_1, m_2, \cdots, m_s\in\mathbb{N}$ and $a_1, a_2, \cdots, a_s\in\mathbb{N}$ satisfy 
\begin{equation}\label{coprime} 
(m_i, m_j)=1, i\neq j, \text{ and } (a_i,m_i)=1, i, j=1, 2, \cdots, s. 
\end{equation} 
Let $m$ and $M_i$ denote 
\[m=m_1m_2\cdots m_s, \quad M_i=\frac{m}{m_i},\quad i=1, 2, \cdots, s\] 
and $\overline{M}_i$ be a positive integer such that 
\begin{equation}\label{inverse} 
\overline{M}_iM_i\equiv1\pmod{m_i}. 
\end{equation} 
Let $r_1, r_2, \cdots, r_{s}$ be positive integers satisfying 
\begin{equation}\label{r} 
r_i\equiv a_iM_i\overline{M}_i\pmod m, i=1, 2, \cdots, s. 
\end{equation} 
Define two sets $\mathcal{A}$ and $\mathcal{B}$ 
\begin{align*} 
\mathcal{A}:=\{r_1j_1+r_2j_2+\cdots +r_sj_s+mkj_{s+1}:\quad 
&0\le j_i\le m_i-1, i=1,2,\cdots,s,\\ 
&0\le j_{s+1}\le l-1 
\} 
\end{align*} 
and 
\[\mathcal{B}:=(mk\mathbb{N}\setminus mkl\mathbb{N}) 
\bigcup\left(\bigcup_{i=1}^s(r_i\mathbb{N}\setminus r_im_i\mathbb{N})\right).\] 
Then 
\[P(\mathcal{A},n)=Q(\mathcal{B},n).\] 
\end{theorem} 
It is obvious that a special case  of Theorem \ref{mainth} that $k=s=1$ yields Theorem \ref{nm}. Theorem \ref{mainth} is essentially  an algorithm that inputs $ k, l$ and $ m_i, a_i, r_i$ for $ i=1, 2, \cdots s$ and outputs the sets $\mathcal{A}$ and $ \mathcal{B}$ such that $P(\mathcal{A},n)=Q(\mathcal{B},n)$. We give an explicit example to illustrate this theorem. 
\begin{example} 
Let $(m_1, m_2, m_3)=(2, 3, 5), (a_1, a_2, a_3)=(1, 1, 1)$ and $l=1$ in Theorem \ref{mainth}. Then $m=30$ and 
\begin{align*} 
&(M_1, M_2, M_3) =(15, 10, 6), \\ 
&(\overline{M}_1,\overline{M}_2,\overline{M}_3) =(1, 1, 1),\\ 
&r_1 \equiv15\pmod{30},\\ 
&r_2 \equiv10\pmod{30},\\ 
&r_3 \equiv6\pmod{30}. 
\end{align*} 
Take $(r_1, r_2, r_3)=(15, 10, 6)$. By Theorem \ref{mainth}, the set $\mathcal{A}$ is 
\begin{align*}\mathcal{A}=\{&6, 10, 12, 15, 16, 18, 20, 21, 22, 24, 25, 26, 27, 28, 31,\\& 32, 33,34, 35, 37, 38, 39, 41, 43, 44, 47, 49, 53, 59\}, 
\end{align*} 
which is the set of exponents of $x$ 
in the expansion of the polynomial 
\[(1+x^{15})(1+x^{10}+x^{20})(1+x^6+x^{12}+x^{18}+x^{24}).\] 
The set $\mathcal{B}$ is 
\begin{align*} 
\mathcal{B} &=(15\mathbb{N}\setminus30\mathbb{N}) 
\cup(10\mathbb{N}\setminus30\mathbb{N})\cup(6\mathbb{N}\setminus30\mathbb{N}) \\ 
&=\{n\in\mathbb{N}: n\equiv15, 10, 20, 6, 12, 18, 24\pmod{30}.\} 
\end{align*} 
Theorem \ref{mainth} shows that the number of partitions of $n$ in which the multiplicities of the parts belong to $\mathcal{A}$ is equal to the number of partitions of $n$ whose parts belong to $\mathcal{B}$. 
\end{example} 
The original proofs of Theorems \ref{Euler}-\ref{nm} are based on generating functions. A beautiful bijection proofs for Theorem \ref{Euler} was presented in \cite{an}. The bijection proofs for Theorems \ref{macmahon}, \ref{andrews} and Theorem \ref{subbarao} were given in \cite{aepr, fs, gu, kdsk}, and for Theorem \ref{nm} were obtained in \cite{ny} and \cite{nm} recently. 
Our proof of Theorem \ref{mainth} is also based on generating functions. However, we believe there exist bijection proofs similar to those for Theorems \ref{Euler}-\ref{nm}. 

Note that the partition functions $P(\mathcal{A}, n)$ and $P(\mathcal{B}, n)$ given in Theorem \ref{mainth} depend on the parameters $ k, l$ and $ m_i, a_i, r_i$, $ i=1, 2, \cdots s$. We shall see  that the generating function for $P(\mathcal{A}, n)$ given by (\ref{generating}) is a quotient of infinite products. A simple observation leads to unexpected congruence properties of $P(\mathcal{A}, n)$. 
\begin{theorem}\label{corollary} 
Let $p(n)$ be the ordinary partition function. Under the assumption of Theorem \ref{mainth}, if there exist $m_i$ and positive integers $c, d$ such that 
\[p(m_in+c)\equiv0\pmod d\] 
for all $n\ge0$, then 
\[P(\mathcal{A}, m_in+a_ic)\equiv0\pmod d\] 
for all $n\ge0$. 
\end{theorem} 
Theorem \ref{corollary} shows that congruences for $p(n)$ yield  congruences for $P(\mathcal{A}, n)$. To obtain such congruences, we need explicit congruences for $p(n)$. 
Weaver \cite{we} found 76,065 explicit congruences for the partition function of the form $p(mn+c)\equiv 0\pmod d$ for primes $13\le d\le 37$. One can use these congruences for $p(n)$ to deduce explicit congruences for $P(\mathcal{A}, n)$. Let us finally mention that Ahlgren \cite{ah} had proven that 
if $d$ is a positive integer coprime to 6, then there exist infinitely 
many arithmetic progressions $mn+c$ (none of which is contained in any other) 
with the property that $p(mn+c)\equiv 0\pmod d$ for all $n\ge0$. Here we use Theorem \ref{corollary} to illustrate a congruence for $P(\mathcal{A}, n)$ given in Example 1.8. 
\begin{example} 
Note that $m_3=5$ and $a_3=1$ in Example 1.8. Recall that the famous Ramanujan's congruence for $p(n)$ modulo $5$ states that 
\[p(5n+4)\equiv0\pmod5, n\ge0.\] 
Thus Theorem \ref{corollary} shows that 
\[p(\mathcal{A},5n+4)\equiv0\pmod5\] 
for all $n\ge0$. 
\end{example}
We shall prove Theorem \ref{mainth} and Theorem \ref{corollary} in Section 2 and present a natural generalization of Euler's theorem in Section 3.

\section{Proofs of Theorem \ref{mainth} and Theorem \ref{corollary}} 
In this section we shall prove Theorem \ref{mainth} and Theorem \ref{corollary}. The following theorem shows that the existence of such partition theorems as Theorems \ref{Euler}-\ref{mainth} are essentially equivalent to the existence of some (finite) geometric series identities. 
\begin{theorem}\label{lemma} 
Let $A \subseteq\mathbb{N}$ be a set of natural numbers and $B \subseteq\mathbb{N}$ be a finite set of natural numbers. Let $\{\delta_b\in\mathbb{N}: b\in B\}$ be a set associated with $B$. Suppose that there is an identity of  (finite) geometric series 
\begin{equation}\label{poly} 
1+\sum_{a\in A}x^a=\prod_{b\in B}\frac{1-x^{b\delta_b}}{1-x^b}. 
\end{equation} 
Define $\mathcal{A}:=A$ and 
\begin{equation}\label{B} 
\mathcal{B}:=\bigcup_{b\in B}(b\mathbb{N}\setminus b\delta_{b}\mathbb{N}). 
\end{equation} 
If for any $b, b'\in B$ with $b\neq b'$, 
\begin{equation}\label{set} 
(b\mathbb{N}\setminus b\delta_{b}\mathbb{N})\bigcap (b'\mathbb{N}\setminus b'\delta_{b'}\mathbb{N})=\emptyset, 
\end{equation} 
then 
\[P(\mathcal{A},n)=Q(\mathcal{B},n).\] 
Here, as $\delta_b$ tends to $\infty$, we need to replace $x^{b\delta_b}$ by 0 in (\ref{poly}), and replace $b\delta_{b}\mathbb{N}$ by $\emptyset$ in (\ref{B}) and (\ref{set}). 
\end{theorem} 
\begin{proof} 
Replacing $x$ with $q^n$ in (\ref{poly}) and taking product over $n=1, 2,\cdots$, we obtain 
\begin{align*} 
\sum_{n=0}^{\infty}P(\mathcal{A},n)q^n&= 
\prod_{n=1}^{\infty}\bigg(1+\sum_{a\in A}q^{an}\bigg)\\ 
&=\prod_{n=1}^{\infty}\prod_{b\in B}\frac{1-q^{b\delta_bn}}{1-q^{bn}}\\ 
&=\prod_{b\in B}\prod_{i_b=1}^{\delta_b-1}\prod_{n=1}^{\infty}\frac{1}{1-q^{b\delta_b n+bi_b}}\\ 
&=\sum_{n=0}^{\infty}Q(\mathcal{B},n)q^n, 
\end{align*} 
where, since the cardinality of $B$ is finite, the infinite products above are convergent for $|q|<1$ (see \cite[p.5]{an}). 
\end{proof} 
\begin{remark} 
Theorem \ref{lemma} gives a quick proof of Theorems 1.1-1.6. For example, Theorem \ref{macmahon} follows from the identity 
\[1+x^2+x^3+\cdots=1+\sum_{i=2}^\infty x^i=\frac{1}{1-x^2}\cdot\frac{1-x^6}{1-x^3}.\] 
Theorem \ref{andrews} follows from the identity 
\[1+\sum_{i=0}^\infty x^{2r+1+2i}=\frac{1}{1-x^2}\cdot\frac{1-x^{4r+2}}{1-x^{2r+1}}.\] 
Theorem \ref{subbarao} follows from the identity 
\[\sum_{i=0}^{m-1} x^{2i}+\sum_{i=0}^{m-1} x^{2i+2r+1}=\frac{1-x^{2m}}{1-x^2}\cdot\frac{1-x^{4r+2}}{1-x^{2r+1}}.\] 
Theorem \ref{nm} follows from 
\[\sum_{i=0}^{m-1}\sum_{j=0}^{p-1}x^{pi+(pr+a)j} 
=\frac{1-x^{pm}}{1-x^p}\cdot\frac{1-x^{p(pr+a)}}{1-x^{pr+a}}.\] 
\end{remark} 
We now apply Theorem \ref{lemma} to prove Theorem \ref{mainth}. 
\begin{proof}[Proof of Theorem \ref{mainth}] 
We expand the following the (finite) geometric series 
\begin{align}\label{main} 
\bigg(\prod_{i=1}^{s}\frac{1-x^{m_ir_i}}{1-x^{r_i}}\bigg) 
\bigg(\frac{1-x^{mkl}}{1-x^{mk}}\bigg) 
&=\bigg(\prod_{i=1}^{s} 
\left(\sum_{j_i=0}^{m_i-1}x^{r_ij_i}\right)\bigg)\left(\sum_{j_{s+1}=0}^{l-1}x^{mkj_{s+1}}\right)\\ 
&=\sum_{j_1=0}^{m_1-1}\sum_{j_2=0}^{m_2-1}\cdots \sum_{j_s=0}^{m_s-1}\sum_{j_{s+1}=0}^{l-1}x^{r_1j_1+r_2j_2+\cdots r_sj_s+mkj_{s+1}}\notag. 
\end{align} 
We first show the exponents of $x$ above are distinct, which means that the set $\mathcal{A}$ is not a multiset. To see this, suppose that \begin{equation}\label{distinct} 
r_1j_1+r_2j_2+\cdots r_sj_s+mkj_{s+1}=r_1j_1'+r_2j_2'+\cdots r_sj_s'+mkj_{s+1}', 
\end{equation} 
where $0\le j_{s+1}, j_{s+1}'\le l-1, 0\le j_i, j_i'\le m_i-1, i=1,2,\cdots,s$. 
In view of (\ref{inverse}) and (\ref{r}), we see that for each $i$ with $1\le i\le s$, 
\begin{equation}\label{ri} 
\left\{\begin{array}{l} 
r_i\equiv a_i\pmod{m_i}, \\ 
r_{i}\equiv0\pmod {m_{i'}} \text{ for any } i' \text{ with } i'\neq i, 1\le i'\le s. 
\end{array} 
\right. 
\end{equation} 
By considering modulo $m_i$ on both sides of (\ref{distinct}), we find 
\[a_ij_i\equiv a_ij_i'\pmod{m_i}.\] 
Since $(a_i, m_i)=1$ by (\ref{coprime}), it follows that $j_i=j_i'$. Inserting this fact into (\ref{distinct}) yields $j_{s+1}=j_{s+1}'$. 

Next we show the residue classes in $\mathcal{B}$ are disjoint, i.e. (\ref{set}) holds. 
Notice that $d\in mk\mathbb{N}\setminus mkl\mathbb{N}$ implies $m_i\mid d$ for any $i$ with $1\le i\le s$. If $d\in r_i\mathbb{N}\setminus r_im_i\mathbb{N}$, then $d$ is of the form $r_i(m_in+j_i)$ for some $n\ge0$ and $j_i$ with $1\le j_i\le m_i-1$. Thus by (\ref{ri}) we have 
\[d\equiv r_ij_i\equiv a_ij_i\pmod{m_i}.\] 
Again, by (\ref{coprime}) we have $(a_i, m_i)=1$. Since $1\le j_i\le m_i-1$, we find that $m_i\nmid d$. Therefore 
\[(mk\mathbb{N}\setminus mkl\mathbb{N})\bigcap(r_i\mathbb{N}\setminus r_im_i\mathbb{N})=\emptyset, i=1, 2,\cdots, s.\] 
We also note that for each $i'$ with $1\le i'\le s, i'\neq i$, if 
\[d\in r_{i'}\mathbb{N}\setminus r_{i'}m_{i'}\mathbb{N},\] 
then $r_{i'}\mid d$. But (\ref{ri}) implies $m_i\mid r_{i'}$, hence $m_i|d$. Thus 
\[(r_i\mathbb{N}\setminus r_im_i\mathbb{N})\bigcap(r_{i'}\mathbb{N}\setminus r_{i'}m_{i'}\mathbb{N})=\emptyset, i'\neq i.\] 
Therefore the residue classes in $\mathcal{B}$ are disjoint. Theorem \ref{mainth} follows from Theorem \ref{lemma} immediately. 
\end{proof} 
\begin{remark}\label{ma} 
The Chinese remainder theorem shows that the set 
\[\{r_1j_1+r_2j_2+\cdots +r_sj_s: 0\le j_i\le m_i, i=1, 2, \cdots, s\}\] is a complete residue system modulo $m$. This implies that if this system is not the least positive residue system modulo $m$, then, as $k=1$ and $l$ tends to $\infty$ in Theorem \ref{mainth}, the set $\mathcal{A}$ differs from the set of natural numbers $\mathbb{N}$ by a finite set $\mathcal{A'}$. In fact, this finite set is 
\begin{align*} 
\mathcal{A'}=\bigcup_{i=1}^s\bigcup_{j_i=0}^{m_i-1}\{a\in\mathbb{N}:\quad &1\le a\le r_1j_1+r_2j_2+\cdots +r_sj_s-1, \\ 
&a\equiv r_1j_1+r_2j_2+\cdots +r_sj_s\pmod m\}. 
\end{align*} 
In this case Theorem \ref{mainth} reduces to 
\[P(\mathbb{N}\setminus\mathcal{A'};n)=Q(\mathcal{B};n),\] 
which is a real MacMahon-type partition theorem (Theorem \ref{macmahon}). 
\end{remark} 
We give an explicit example to illustrate MacMahon-type partition theorem. 
\begin{example} 
In Theorem \ref{mainth}, let $(m_1, m_2)=(2, 3) $, $ k=1$ and $l=\infty$. Then we have 
\[m=6, (M_1, M_2)=(3, 2), (\overline{M}_1, \overline{M}_2)=(1, 2)\] 
and 
\[ r_1\equiv a_1M_1\overline{M}_1\equiv3a_1\pmod6,\] 
\[ r_2\equiv a_2M_2\overline{M}_2 \equiv 4a_2 \pmod6.\] 
Take $(a_1, a_2)=(1, 1)$ and $ (r_1, r_2)=(3, 4)$. We have the identity 
\begin{align*} 
\frac{1}{1-x^6}\cdot\frac{1-x^6}{1-x^3}\cdot\frac{1-x^{12}}{1-x^4} 
&=\bigg(\sum_{i=0}^\infty x^{6i}\bigg)(1+x^3)(1+x^4+x^8)\\ 
&=\sum_{i=0}^\infty (x^{6i}+x^{6i+3}+x^{6i+4}+x^{6i+7}+x^{6i+8}+x^{6i+11})\\ 
&=\sum_{\substack{i=1\\i\neq 1, 2, 5}}^{\infty}x^i. 
\end{align*} 
Thus $\mathcal{A}'=\{1, 2, 5\}$ and Theorem \ref{mainth} gives 
\begin{align*} 
\mathcal{A}&=\mathbb{N}\setminus\mathcal{A}'=\{a\in \mathbb{N}: a\neq 1, 2, 5\},\\ 
\mathcal{B}&=\{6\mathbb{N}\bigcup(3\mathbb{N}\setminus6\mathbb{N}) 
\bigcup(4\mathbb{N}\setminus12\mathbb{N})\}\\ 
&=\{b\in\mathbb{N}: b\equiv0, 3, 4, 6, 8, 9\pmod{12}\} 
\end{align*} and 
$P(\mathcal{A};n)=Q(\mathcal{B};n)$. 
\end{example} 
Now we begin to prove Theorem \ref{corollary}. 
\begin{proof}[Proof of Theorem \ref{corollary}] By Theorem \ref{lemma} 
we have 
\begin{equation}\label{generating} 
\sum_{n=0}^{\infty} P(\mathcal{A};n)q^n 
=\prod_{n=1}^{\infty}\frac{1-q^{mkln}}{1-q^{mkn}}\cdot 
\prod_{i=1}^{s}\bigg(\prod_{n=1}^{\infty}\frac{1-q^{r_im_in}}{1-q^{r_in}}\bigg). 
\end{equation} 
We see from (\ref{ri}) that $m_i\mid m$ and $m_i|r_j$ for all $j\neq i, 1\le j\le s$, which implies each powers of $q$ in the expansion of infinite products above is a multiple of $m_i$ except that in the expansion of $\prod_{n=1}^{\infty}(1-q^{r_in})$. Therefore we can rewrite (\ref{generating}) as 
\begin{align*} 
\sum_{n=0}^{\infty} P(\mathcal{A};n)q^n 
&=\prod_{n=1}^{\infty}\frac{1}{1-q^{r_in}}G(q^{m_i}), 
\end{align*} 
where the power series $G(q)$ is clear by (\ref{generating}). Suppose that $G(q)$ has an expansion $G(q)=\sum_{v=0}^{\infty}g(v)q^{v}$. Using the generating function for $p(n)$, i.e.
\[\prod_{n=1}^{\infty}\frac{1}{1-q^{r_in}}=\sum_{u=0}^{\infty}p(u)q^{r_iu},\] 
we deduce that 
\begin{align*} 
\sum_{n=0}^{\infty} P(\mathcal{A};n)q^n 
&=\bigg(\sum_{u=0}^{\infty}p(u)q^{r_iu}\bigg)\bigg(\sum_{v=0}^{\infty}g(v)q^{m_iv}\bigg)\\ 
&=\sum_{n=0}^\infty \bigg(\sum_{r_iu+m_iv=n}p(u)g(v)\bigg)q^n. 
\end{align*} 
Hence 
\begin{equation}\label{gf} 
P(\mathcal{A};n)=\sum_{r_iu+m_iv=n}p(u)g(v). 
\end{equation} 
If $n\equiv a_ic\pmod{m_i}$, then 
\[n=r_iu+m_iv\equiv a_iu\equiv a_ic\pmod {m_i}.\] 
Since $(a_i,m_i)=1$, we have $u\equiv c\pmod{m_i}$. But by assumption $p(m_in+c)\equiv0\pmod d$ for all $n\ge0$, we deduce from (\ref{gf}) that 
\[P(\mathcal{A};m_in+a_ic)\equiv0\pmod d \] 
for all $n\ge0$. 
\end{proof} 
\section{ Partition theorems and complete residue systems} 
Theorem \ref{mainth} illustrates  that the set $\mathcal{A}$ in $P(\mathcal{A}, n)$ can be generated from a complete residue system (refer to  Remark \ref{ma}). This idea is useful to construct  more partition identities analogous to Theorem \ref{mainth}.  Let $m_1, m_2, \cdots, m_s$ be $s$ positive integers. We observe that the set
\[\{j_1+m_1j_2+m_1m_2j_3+\cdots+m_1m_2\cdots m_{s-1}j_{s}: j_i \in \mathbb{Z}/m_i\mathbb{Z}, i=1, 2, \cdots, s\}\]
is a complete residue system of modulo $m_1m_2\cdots m_s$. In this section, we shall use this fact to give another partition theorem.
\begin{theorem}\label{mainth1} 
Suppose that $l, s, m_1, m_2, \cdots, m_s, r_1, r_2, \cdots, r_s$ are positive integers and satisfy 
\begin{equation}\label{rr} 
r_1\mid r_{2},\, r_2\mid r_3,\, \cdots, r_{s}\mid r_{s+1}. 
\end{equation} 
Define the sets $\mathcal{A}$ and $\mathcal{B}$ 
\begin{align*}\mathcal{A}=\{r_1j_1&+m_1r_2j_2+\cdots+ m_1m_2\cdots m_{s-1}r_sj_s\\ 
&+ m_1m_2\cdots m_{s}r_{s+1}j_{s+1}: 0\le j_{s+1}\le l-1, 0\le j_i\le m_i-1, i=1,2,\cdots,s\} 
\end{align*} 
and 
\begin{align*} 
\mathcal{B}=&(r_1\mathbb{N}\setminus m_1r_1\mathbb{N})\cup (m_1r_2\mathbb{N}\setminus m_1m_2r_2\mathbb{N})\cup\cdots 
\\ 
&\cup(m_1m_2\cdots m_{s-1}r_s\mathbb{N}\setminus m_1m_2\cdots m_{s}r_s\mathbb{N})\\ 
&\cup (m_1m_2\cdots m_{s}r_{s+1}\mathbb{N}\setminus m_1m_2\cdots m_{s}r_{s+1}l\mathbb{N}). 
\end{align*} 
Then \[P(\mathcal{A};n)=Q(\mathcal{B};n).\] 
\end{theorem} 
Note that  when $r_1=r_2=\cdots=r_{s+1}=1$ and $m_1=m_2=\cdots=m_s=1$, Theorem \ref{mainth1} reduces to Theorem \ref{Glaisher} for $d=l$. Therefore Theorem \ref{mainth1} may also be seen as a natural generalization of Euler's theorem.
\begin{proof}[Proof of Theorem \ref{mainth1}] 
We expand the (finite) geometric series and obtain 
\begin{align*}\label{main} 
&\frac{1-x^{m_1r_1}}{1-x^{r_1}}\frac{1-x^{m_1m_2r_2}}{1-x^{m_1r_2}}\cdots 
\frac{1-x^{ m_1m_2\cdots m_{s-1}m_sr_s}}{1-x^{m_1m_2\cdots m_{s-1}r_s }}\frac{1-x^{ m_1m_2\cdots m_{s}r_{s+1}l}}{1-x^{m_1m_2\cdots m_{s}r_{s+1}}}\\ 
&=\bigg(\sum_{j_1=0}^{m_1-1}x^{r_1j_1}\bigg)\bigg(\sum_{j_2=0}^{m_2-1}x^{m_1r_2j_2}\bigg) 
\cdots\bigg(\sum_{j_s=0}^{m_s-1}x^{ m_1m_2\cdots m_{s-1}r_sj_s}\bigg) 
\\ 
&\hskip6cm\times\bigg(\sum_{j_{s+1}=0}^{l-1}x^{ m_1m_2\cdots m_{s}r_{s+1}j_{s+1}}\bigg)\\ 
&=\sum_{j_1=0}^{m_1-1}\cdots\sum_{j_s=0}^{m_s-1}\sum_{j_{s+1}=0}^{l-1} 
x^{r_1j_1+m_1r_2j_2+\cdots+ m_1m_2\cdots m_{s-1}r_sj_s+ m_1m_2\cdots m_{s}r_{s+1}j_{s+1}}. 
\end{align*} 
We claim that the exponents of $x$ above are distinct, that is, the set $\mathcal{A}$ is not a multiset. To see this, suppose that 
\begin{equation}\label{distinct1} 
r_1j_1+m_1r_2j_2+\cdots+ m_1m_2\cdots m_{s-1}r_sj_s+m_1m_2\cdots m_{s}r_{s+1}j_{s+1} 
\end{equation} 
\begin{equation*} 
=r_1j_1'+m_1r_2j_2'+\cdots+ m_1m_2\cdots m_{s-1}r_sj_s'+m_1m_2\cdots m_{s}r_{s+1}j_{s+1}', 
\end{equation*} 
where $0\le j_{s+1}, j_{s+1}'\le l-1, 0\le j_i, j_i'\le m_i-1, i=1,2,\cdots,s$. By (\ref{rr}) 
dividing by $r_1$ on both sides of (\ref{distinct1}) and considering modulo $m_1$, we see that $j_1=j_1'$. Then deleting $r_1j_1$, dividing by $m_1r_2$ on both sides of (\ref{distinct1}) and considering modulo $m_2$, we deduce that $j_2=j_2'$. Continuing in this way, we find that $j_i=j_i'$ for $i=0, 1, \cdots, s$. In view  of (\ref{rr}), it is obvious that the residue classes in $\mathcal{B}$ are disjoint. Theorem \ref{mainth1} follows from Theorem \ref{lemma}. 
\end{proof} 
Analogous to Theorem \ref{corollary}, with the help of the generating function, we can show that the partition function  $P(\mathcal{A};n)$ in Theorem \ref{mainth1} also possesses  a similar congruence relation with the  partition function $p(n)$.
\begin{theorem}\label{corollary1} 
Let $p(n)$ be the partition function. Under the assumption of Theorem \ref{mainth1}, if there exist positive integers $c, d$ such that 
\[p(m_1n+c)\equiv0\pmod d\] 
for all $n\ge0$, then 
\[P(\mathcal{A}, m_1n+c)\equiv0\pmod d\] 
for all $n\ge0$. 
\end{theorem} 
\begin{proof} 
From the proof of Theorem \ref{mainth1} we see easily that the generating function for $P(\mathcal{A},n)$ is 
\begin{align*} 
&\quad\sum_{n=0}^{\infty}P(\mathcal{A},n)q^n\\ 
&= 
\prod_{n=1}^{\infty}\frac{1-q^{r_1m_1n}}{1-q^{r_1n}}\frac{1-q^{r_2m_1m_2n}}{1-q^{r_2m_1n}}\cdots 
\frac{1-q^{r_s m_1m_2\cdots m_{s-1}m_sn}}{1-q^{r_s m_1m_2\cdots m_{s-1}n}}\frac{1-q^{ r_{s+1}m_1m_2\cdots m_{s}ln}}{1-q^{r_{s+1}m_1m_2\cdots m_{s}n}} \\ 
&=\prod_{n=1}^{\infty}\frac{1}{1-q^{r_1n}}\cdot G_1(q^{r_1m_1}), 
\end{align*} 
say.  
The remaining arguments are similar to those in the proof of Theorem \ref{corollary} and are therefore omitted. 
\end{proof} 

We have seen from Theorems \ref{mainth} and \ref{mainth1} that   complete residue systems  are highly useful in constructing partition identities of the form (\ref{identity}). Certainly, one can obtain more partition theorems in this manner.  
Finally,  we emphasize  that any  bijection proofs of Theorems \ref{mainth} and \ref{mainth1} would be very interesting. 
\vskip1cm
\noindent{\bf Acknowledgements} This work was  supported by  the national natural science foundation of China (11771121). 
 

\begin{thebibliography}{99} 
\normalsize 
\baselineskip=17pt 
\bibitem{ah} S. Ahlgren, Distribution of the partition function 
modulo composite integers $M$, Math. Ann. 318(2000),795--803. 
\bibitem{ald}  H. L. Alder, Partition identities—from Euler to the present, Amer. Math. Monthly 76 (1969), 733--746.
\bibitem{an} G. E. Andrews K. Eriksson, Integer partitions, Cambridge Univ. Press, 
Cambridge (2004). 
\bibitem{an1} G. E. Andrews, A generalisation of a partition theorem of MacMahon, J. 
Combin. Theory, 3(1967), 100--101. 
\bibitem{aepr} G. E. Andrews, H. Eriksson, F. Petrov, D. Romik, Integrals, partitions 
and MacMahon's Theorem, J. Combin. Theory, Series A, 114(2007), 545--554. 
\bibitem{fs} S. Fu, J. A. Sellers, Bijective proofs of partition identities of MacMahon, Andrews, and Subbarao, Electron. J. Combin., 21(2)(2014), 
1--9. 
\bibitem{gu} H. Gupta, A partition theorem of Subbarao, Can. Math. Bull., 17(1)(1974), 121--123. 
\bibitem{kdsk} M. R. R. Kanna, B. N. Dharmendra, G. Sridhara, R. P. Kumar, Generalized bijective proof of the partition identity of M.V. Subbarao, Int. Math. Forum, 8(5)(2013), 215--222. 
\bibitem{ma} P. A. MacMahon, Combinatory Analysis(Vol. 2), Cambridge University Press, 
1916. 
\bibitem{ny} D. Nyirenda, A note on Andrews-MacMahon theorem, https://arxiv.org/ 
pdf/2212.13926.pdf 
\bibitem{nm} D. Nyirenda, B. Mugwangwavari, On generalizations of theorems of MacMahon 
and Subbarao. Ann. Comb.  27(2023), 373--386 . 
\bibitem{su} M. V. Subbarao, On a partition theorem of MacMahon-Andrews, Proc. 
Amer. Math. Soc., 27(3)(1971), 449--450. 
\bibitem{we} R. L. Weaver, New congruences for the partition function, Ramanujan J., 5(2001) 53--63. 
\end{thebibliography}
\end{document}